\documentclass{amsart}
\usepackage{hyperref}

\usepackage{tikz}
\makeatletter

\newfont{\gothic}{eufm10 scaled 1100}

\theoremstyle{plain}    
\newtheorem{thm}{Theorem}[section]
\numberwithin{equation}{section} %% Comment out for sequentially-numbered
\numberwithin{figure}{section} %% Comment out for sequentially-numbered
\theoremstyle{plain}    
\newtheorem{cor}[thm]{Corollary} %%Delete [thm] to re-start numbering
\theoremstyle{plain}    
\newtheorem{conj}[thm]{Conjecture} %%Delete [thm] to re-start numbering
\theoremstyle{plain}    
\theoremstyle{plain}
\newtheorem{lem}[thm]{Lemma} %%Delete [thm] to re-start numbering
\theoremstyle{plain}    
\newtheorem{prop}[thm]{Proposition} %%Delete [thm] to re-start numbering
\theoremstyle{plain}    
\newtheorem{Def}[thm]{Definition} %%Delete [thm] to re-start numbering
\theoremstyle{remark}
\newtheorem{rem}[thm]{Remark}
\theoremstyle{remark}    

\theoremstyle{remark}    
\newtheorem{exm}[thm]{Example}

\usepackage[arrow,matrix,curve,ps]{xy}
\xyoption{dvips}
\CompileMatrices

\makeatother

\begin{document}

\title{K\"ahler packings and Seshadri constants on projective complex surfaces}

%\date{\today}

\author{Thomas Eckl}



%\thanks{}

\address{Thomas Eckl, Department of Mathematical Sciences, The University of Liverpool, Mathematical
               Sciences Building, Liverpool, L69 7ZL, England, U.K.}

\email{thomas.eckl@liv.ac.uk}

\urladdr{http://pcwww.liv.ac.uk/~eckl/}

\maketitle

\begin{abstract}
In analogy to the relation between symplectic packings and symplectic blow ups we show that multiple point Seshadri constants on projective complex surfaces can be calculated as the supremum of radii of multiple K\"ahler ball embeddings. We exemplify this connection on toric surfaces, also discussing how toric moment maps reflect the packing.
\end{abstract}

%\tableofcontents

\pagestyle{myheadings}
\markboth{THOMAS ECKL}{K\"AHLER PACKINGS AND SESHADRI CONSTANTS}

\setcounter{section}{-1}

\section{Introduction}

\noindent Symplectic Topology, searching for global properties of symplectic manifolds, is a rather new branch of the old and venerable study of symplectic structures (see the in-depth treatise of McDuff and Salamon, \cite{McDS95}). One of its most striking successes is the analysis and solution of several symplectic packing problems: How large can symplectic balls of a given number be when disjointly embedded in a given symplectic manifold? These questions are especially attractive as they exhibit the fundamental nature of symplectic structures: local flexibility vs (sometimes) global strictness. In particular, symplectic packing is not so strict as Euclidean (that is, distance-and-angle preserving) packing, leading to questions like the Kepler conjecture on ball packings (and its solution by Hales \cite{Hal05, Hal12}), but may be not so flexible as only volume-preserving packing. So some symplectic packing problems reveal obstacles to packings without gaps, whereas other packings are possible without gaps. 

\noindent One of the most prominent series of such symplectic packing problems asks how large symplectic balls of a given number can be when disjointly embedded in the complex projective plane $\mathbb{CP}^2$. In more details, consider balls $B_0(r) \subset \mathbb{R}^4$ of radius $r$ centered in $0$ together with the symplectic form $\omega_{\mathrm{std}}$ obtained by restricting the standard symplectic form on $\mathbb{R}^4$. If $\coprod_{q=1}^k B_0(r_q)$ denotes the disjoint union of $k$ of these balls, with possibly different radii, and $\omega_{\mathrm{FS}}$ denotes a Fubini-Study K\"ahler form on $\mathbb{CP}^2$ , then a symplectic packing of $\mathbb{CP}^2$ with $k$ symplectic balls is defined as a symplectic embedding
\[ \iota: \coprod_{q=1}^k B_0(r_q) \hookrightarrow \mathbb{CP}^2, \]
that is, $\iota$ is a smooth embedding such that $\iota^\ast \omega_{\mathrm{FS}}|_{B_0(r_q)} = \omega_{\mathrm{std}}$. The symplectic packing problem asks on conditions on the radii $r_q$ such that such a symplectic packing exist, and also how it can be explicitely constructed.

\noindent McDuff and Polterovich \cite{McDP94} connected this problem first to symplectic blow ups and then to Algebraic Geometry: They showed that a symplectic packing with balls of radii $r_q$ is only possible if on $\sigma: X \rightarrow \mathbb{CP}^2$, the blow up of $\mathbb{CP}^2$ (considered as a complex manifold) in $k$ points $x_1, \ldots, x_k$, there exists a symplectic form representing the cohomology class of $\sigma^\ast l - \pi \sum_{q=1}^k r_q^2 e_q$, where $l$ and the $e_q$ are Poincar\'e dual to a line $L \subset \mathbb{CP}^2$ and the exceptional divisors $E_q = \sigma^{-1}(x_q)$. Then they proved that as long as $k \leq 8$ the only obstacles to the existence of such a symplectic form are the same as for the existence of a K\"ahler form in this cohomology class, namely $(-1)$-curves on $X$. Thus, these symplectic packing problems are merged with the algebraic-geometric theory of del Pezzo surfaces already extensively studied in the 19th century (see \cite{M74} for a survey and results). 

\noindent Next, Biran \cite{Bir97} was able to prove that $(-1)$-curves remain the only obstacles for symplectic packings with $k \geq 9$ balls, and as a consequence he showed the symplectic analogue of a celebrated algebraic-geometric conjecture named after Nagata, who came across it when solving Hilbert's Fourteenth Problem \cite{Nag59}.
\begin{conj}[Nagata]
With notations as above and $k \geq 9$, there is a K\"ahler form representing the cohomology class $\sigma^\ast l - \epsilon \cdot \sum_{q=1}^k e_q$ for all $\epsilon < \frac{1}{\sqrt{k}}$ if the blown-up points $x_1, \ldots, x_k$ are chosen sufficiently general.
\end{conj}

\noindent Note that this is the K\"ahler version of a purely algebraic-geometric statement:
\begin{conj}[Nagata, algebraic-geometric version]
Let $C = \{F = 0\}$ be an irreducible algebraic curve in $\mathbb{CP}^2$, given by an irreducible homogeneous polynomial $F = F(X,Y,Z)$ of degree $d$ in three homogeneous variables $X, Y, Z$ and with multiplicity $m_q$ in the point $x_q$ (that is, $m_q$ is the lowest degree of a non-vanishing term in the Taylor series expansion of $F$ around $x_q$). If the points $x_1, \ldots, x_k \in \mathbb{CP}^2$ are chosen sufficiently general then
\[ \sqrt{k} d \geq \sum_{q=1}^k m_q. \] 
\end{conj}

\noindent The equivalence of these two conjectures follows from the fact that K\"ahler forms representing an integral cohomology class are curvature forms of hermitian metrics on an ample line bundle (that is Kodaira's embedding theorem \cite[Thm.III.4.1, III.4.6]{Wel80}) and that intersection numbers of ample cohomology classes with algebraic curves are always positive (that is the easy half of Nakai-Moishezon's Ampleness Criterion \cite[Thm.1.2.23]{LazPAG1}). The Nakai-Moishezon Criterion is applied on the strict transform $\overline{C}$ on $X$, that is the inverse image of $C$ under the blow-up map $\sigma$ without the exceptional divisors $E_q$ (intersected $m_q$ times by $\overline{C}$). 

\noindent More on Biran's proof and on what the symplectic methods tell us for the algebraic situation (in particular, why they cannot be used so easily for Nagata's Conjecture) can be found in Biran's lucid survey \cite{Bir01}. 

\noindent The aim of this note is to show that the algebraic conjecture of Nagata is equivalent to a more restricted packing problem, namely a \textit{K\"ahler packing problem}.

\noindent We work in a more general setting: $V$ is assumed to be a $2$-dimensional projective complex manifold (surface for short), $x_1, \ldots, x_k \in V$ distinct points on $V$ and $L$ an ample line bundle $L$ on $V$. Sometimes we interpret $L$ also as a divisor on $V$. Let $\sigma: \widetilde{V} \rightarrow V$ be the blow up of the $k$ points $x_1, \ldots, x_k$, with exceptional divisors $E_q = \sigma^{-1}(x_q)$.
\begin{Def}
The multi-point Seshadri constant $\epsilon(L;x_1, \ldots, x_k)$ is defined as 
\[ \sup \{ \epsilon > 0: \mathrm{A\ multiple\ of\ } \sigma^\ast L - \epsilon \sum_{q=1}^k E_q\ \mathrm{is\ an\ ample\ divisor}\}. \]
\end{Def}

\noindent Nagata's Conjecture predicts that $\epsilon(L; x_1, \ldots, x_k) = \frac{1}{\sqrt{k}}$ on $V = \mathbb{CP}^2$, where $L \subset \mathbb{CP}^2$ is a line and $x_1, \ldots, x_k \in \mathbb{CP}^2$ are points in sufficiently general position. Seshadri constants were busily investigated in Algebraic Geometry during the last years, as a measure of local positivity (see e.g. \cite[Ch.5]{LazPAG1}).

\noindent K\"ahler packings can be defined on arbitrary K\"ahler manifolds: 
\begin{Def}
Let $(V, \omega)$ be a $n$-dimensional K\"ahler manifold with K\"ahler form~$\omega$. Then a holomorphic embedding 
\[ \phi = \coprod_{q=1}^k \phi_q: \coprod_{q=1}^k B_0(r_q) \rightarrow V \]
is called a K\"ahler embedding of $k$ disjoint complex balls in $\mathbb{C}^n$ centered in~$0$, of radii~$r_q$, if $\phi_q^\ast(\omega) = \omega_{\mathrm{std}}$, the standard K\"ahler form on $\mathbb{C}^n$ restricted to $B_0(r_q)$.
\end{Def}

\noindent  Using the setting above we prove in Section~\ref{packing-sec}:
\begin{thm}   \label{packing-thm}
For all $0 < \epsilon < \epsilon(L;x_1, \ldots, x_k)$ there exists a K\"ahler form $\omega$ on $V$ representing the first Chern class $c_1(L)$ of $L$ and a K\"ahler packing (wrt $\omega$) of $k$ disjoint balls of radii $\sqrt{\frac{\epsilon}{\pi}}$ into $V$. Vice versa, if for $\epsilon > 0$ there exists such a K\"ahler packing wrt a K\"ahler form $\omega$ representing $c_1(L)$ then $\epsilon < \epsilon(L;x_1, \ldots, x_k)$.
\end{thm}

\noindent If we provide the balls with the Fubini-Study form $\omega_{\mathrm{FS}}$ restricted to $B_0(r) \subset \mathbb{C}^2 \subset \mathbb{CP}^2$ we can show that the corresponding K\"ahler problem is equivalent to determining the Seshadri constant, too:  

\begin{cor} \label{FS-packing-cor}
For all $0 < \epsilon < \epsilon(L;x_1, \ldots, x_k)$ there exists a K\"ahler form $\omega$ on $V$ representing the first Chern class $c_1(L)$ of $L$ and (for $R > 0$ arbitrarily large) a K\"ahler packing 
\[ \coprod_{i=1}^k (B_0(R), \frac{\epsilon}{\pi} \cdot \omega_{\mathrm{FS}}) \rightarrow (V, \omega). \]
Vice versa, if for $\epsilon > 0$ there exists such a K\"ahler packings wrt a K\"ahler form $\omega$ representing $c_1(L)$ and for $R$ arbitrary large then $\epsilon < \epsilon(L;x_1, \ldots, x_k)$.
\end{cor}

\noindent In \cite{WN15a, WN15b} Witt Nystr\"om proved similar results in arbitrary dimension, but only for one point.

\noindent In Section~\ref{toric-sec} we discuss K\"ahler packings on a smooth projective complex toric surface $V$ when the blown-up points $x_1, \ldots, x_k$ are fixed points of the torus action on $V$. In this situation it is easy to calculate the multi-point Seshadri constant $\epsilon(L; x_1, \ldots, x_k)$, generalizing the case $k = 1$ discussed in \cite{DiR99}, \cite{BDH+09}, see Cor.~\ref{toric-Sesh-cor}. It is also possible to approximate K\"ahler packings by Fubini-Study balls using K\"ahler forms induced by global sections of large enough multiples of $L$ stable under the torus action, see Thm.~\ref{approx-packing-thm}. Actually, choosing global sections carefully the same idea works for general surfaces $V$. The additional tool needed to prove Thm.~\ref{packing-thm} and Cor.~\ref{FS-packing-cor} is the symplectic blow-up and blow-down procedure developed by McDuff and Polterovich \cite[\S 5]{McDP94}, to glue in flat resp.\ Fubini-Study balls.  

\noindent Finally, we show in Thm.~\ref{approx-packing-thm} that the toric symplectic moment maps induced by those sections pulled back to the embedded balls approximate the Fubini-Study moment map on a Fubini-Study ball. This gives an interpretation to the change from the toric moment polytope of the line bundle $L$ on $V$ to the toric moment polytope of $\pi^\ast L - \epsilon \sum_q E_q$ on the blow-up of $V$ (see Prop.~\ref{ample-blowup-prop} for an exact statement and Ex.~\ref{toric-ex} for an illustration): The cut-off triangles of the moment polytope are the shadows of the embedded balls under the moment map.

\vspace{0.1cm}

\noindent \textit{Acknowledgements.} The author thanks the anonymous referee for suggestions on how to improve the exposition and to include examples, and David Witt Nystr\"om for discussing his results in \cite{WN15a, WN15b} with the author.

\section{Proof of Theorem~\ref{packing-thm} and Corollary~\ref{FS-packing-cor}}  \label{packing-sec}

\noindent The idea to prove one direction of the Thm.~\ref{packing-thm} is to construct K\"ahler forms on the blow-up of $V$ from global sections of $L^{\otimes m}$ vanishing to higher and higher order in the points $x_1, \ldots, x_k$. If the sections are carefully chosen the vanishing is homogeneous in all directions, and the K\"ahler forms get sufficiently flat around the exceptional divisors $E_q$ over $x_q$ to be able to glue in a standard K\"ahler ball of a radius arbitrarily close to $\sqrt{\frac{\epsilon}{\pi}}$. The main technical tool for the gluing procedure is the symplectic blow down described by McDuff and Polterovich \cite[\S 5.4]{McDP94}. 

\noindent In more details, recall that the standard K\"ahler form $\omega_0$ on $\mathbb{C}^2$ is given in affine holomorphic coordinates $(x,y)$ by $\frac{i}{2}(dx \wedge d\overline{x} + dy \wedge d\overline{y})$, whereas the Fubini-Study K\"ahler form $\tau_0$ on $\mathbb{CP}^1$ is given in homogeneous coordinates $[S:T]$ by $\frac{i}{2\pi}\partial\overline{\partial} \log(S\overline{S}+T\overline{T})$. Note that the latter $(1,1)$-form is well-defined on $\mathbb{CP}^1$ because $S\overline{S}+T\overline{T}$ is homogeneous in $S$ and $T$, and that it represents $c_1(\mathcal{O}_{\mathbb{CP}^1}(P))$ for any point $P \in \mathbb{CP}^1$. More generally, if $s_0, \ldots, s_N$ are sections of a line bundle $L$ on a complex manifold $X$ defining an embedding
\[ X \hookrightarrow \mathbb{CP}^N,\ x \mapsto [S_0:\cdots:S_N] = [s_0(x):\cdots:s_N(x)] \]
(for example, if the $s_i$ span $H^0(X,L)$ and $L$ is very ample)
then we can use this embedding to construct a K\"ahler form on $X$, by restricting the Fubini-Study form $\frac{i}{2\pi}\partial\overline{\partial} \log(\sum_{k=0}^N S_k\overline{S}_k)$ in homogeneous coordinates $[S_0:\cdots:S_N]$ on $\mathbb{CP}^N$ to $X$. We say that this restricted K\"ahler form on $X$ is \textit{induced} by the sections $s_0, \ldots, s_N$. The K\"ahler form can also be seen as the curvature form of the hermitian metric $h$ induced by the sections on $L$, defining the length of the vector $\xi(x)$ for each section $\xi$ of $L$ and point $x$ on $X$ by
\[ \parallel \xi \parallel^2_h := \frac{\xi(x)\overline{\xi(x)}}{\sum_{k=0}^N s_k(x)\overline{s_k(x)}}. \]

\noindent Let $(x,y)$ be local complex coordinates around $x_q \in V$, and denote by $S := \frac{x}{y}, T := \frac{y}{x}$ the induced homogeneous coordinates on the exceptional divisor $E_q \cong \mathbb{CP}^1$. If $U_q(\delta)$ denotes a ball centered in $x_q$ of sufficiently small radius $\delta$, measured according to the coordinates $x,y$, then the tubular neighborhood $\sigma^{-1}(U_q(\delta)) \subset \widetilde{V}$ of $E_q$ is projected to $E_q \cong \mathbb{CP}^1$ by a holomorphic map $p_q$ collapsing the lines in $U_q(\delta)$ through $(0,0)$. Furthermore $\sigma^{-1}(U_q(\delta))$ is covered by two charts with coordinates $(x,t)$ resp. $(s,y)$, with transition maps given by $y = xt$ and $s = 1/t$. Note that the exceptional divisor $E_q$ intersects these charts as the vanishing locus of $x$ resp. $y$. 

\noindent Now assume that $\epsilon \in \mathbb{Q}$. Then for $n > 0$ a sufficiently divisible integer, the line bundle $\widetilde{L}_n := \sigma^\ast(L^{\otimes n}) \otimes \mathcal{O}_{\widetilde{V}}(-n\epsilon \cdot \sum_{q=1}^k E_q)$ is ample.
On $U_q(\delta)$ the line bundle $L^{\otimes n}$ is trivial, hence we can define a hermitian metric $h_0$ on $L^{\otimes n}_{|U_q(\delta)}$ by
\[ \parallel \xi \parallel^2_{h_0} := \frac{\xi(x)\overline{\xi(x)}}{e^{x\overline{x}+y\overline{y}}} \]
with everywhere positive curvature form $\frac{1}{\pi} \omega_0 = \frac{i}{2\pi}(dx \wedge d\overline{x} + dy \wedge d\overline{y})$. If $\sigma^\ast h_0$ denotes the pulled back metric on $\sigma^\ast(L^{\otimes n})$ its curvature form $\sigma^\ast \omega_0$ is everywhere semipositive on $U_q(\delta)$ and positive away from $E_q$.

\noindent The sections of $\mathcal{O}_{\sigma^{-1}(U_q(\delta))}(-n\epsilon \cdot \sum_{q=1}^k E_q)$
given in the coordinates $(x,t)$ of one of the charts covering $\sigma^{-1}(U_q(\delta))$ by 
\[ \sqrt{\binom{n\epsilon}{j}} t^j,\ j = 0, \ldots, n\epsilon, \]
define a hermitian metric $h_q$ on $\mathcal{O}_{\sigma^{-1}(U_q(\delta))}(-n\epsilon \cdot \sum_{q=1}^k E_q)$. Note that the coefficient of $t^j$ allows to rewrite the metric induced by these sections on $\mathcal{O}_{\sigma^{-1}(U_q(\delta))}(-n\epsilon \cdot \sum_{q=1}^k E_q)$ as a power of the metric induced by the sections $1$ and $t$ on $\mathcal{O}_{\sigma^{-1}(U_q(\delta))}(-\sum_{q=1}^k E_q)$. Since the curvature form of $h_q$ is positive on $E_q$ the tensor product $\sigma^\ast h_0 \otimes h_q$ is a hermitian metric $h_{0,q}$ on $\widetilde{L}_{n|\sigma^{-1}(U_q(\delta))}$ with everywhere positive curvature form
\[ \omega_{0,q} := \frac{1}{\pi} \sigma^\ast \omega_0 + n\epsilon \cdot p_q^\ast \tau_0. \]

\noindent \textit{Step 1.} For $n \gg 0$ we can find sections $\sigma_0, \ldots, \sigma_N$ spanning $H^0(\widetilde{V}, \widetilde{L}_n)$ such that the induced hermitian metric $\widetilde{h}$ on $\widetilde{L}_n$ has a positive curvature form $\widetilde{\omega}$ satisfying
\[ \widetilde{\omega}_{|E_q} = n\epsilon \cdot \tau_0\ \mathrm{and\ } \widetilde{\omega}(P) = \frac{1}{\pi} (\sigma^\ast \omega_0)(P) + n\epsilon \cdot (p_q^\ast \tau_0)(P), \]
for all $q=1, \ldots, k$ and all points $P \in E_q$: For $n \gg 0$ the line bundle $\widetilde{L}_n$ is sufficiently ample such that the restriction maps
\begin{eqnarray*}
H^0(\widetilde{V}, \widetilde{L}_n \otimes \mathcal{O}_{\widetilde{V}}(-2\sum_{r=1}^k E_r)) & \rightarrow & \bigoplus_{r=1}^k H^0(E_r, \widetilde{L}_{n|E_r} \otimes \mathcal{O}_{E_r}(-2E_r)) = \\  
    & = & \bigoplus_{r=1}^k H^0(E_r, \mathcal{O}_{E_r}(-(n\epsilon+2) E_r)),
\end{eqnarray*}
\[ H^0(\widetilde{V}, \widetilde{L}_n \otimes \mathcal{O}_{\widetilde{V}}(-\sum_{r=1}^k E_r))  \rightarrow
    \bigoplus_{r=1}^k H^0(E_r, \mathcal{O}_{E_r}(-(n\epsilon+1) E_r)), \]
\[ H^0(\widetilde{V}, \widetilde{L}_n \otimes \mathcal{O}_{\widetilde{V}}(-\sum_{r \neq q} E_r)) \rightarrow 
   \bigoplus_{r \neq q} H^0(E_r, \mathcal{O}_{E_r}(-(n\epsilon+1) E_r)) \oplus H^0(E_q, \mathcal{O}_{E_q}(-n\epsilon E_q)) \]
are surjective for each $q = 1, \ldots, k$, by Serre Vanishing \cite[Thm.1.2.6]{LazPAG1}. For each $q = 1, \ldots, k$ we can thus find 
\begin{itemize}
\item sections in $H^0(\widetilde{V}, \widetilde{L}_n)$ restricting to $S^{n\epsilon+2}, S^{n\epsilon+1}T, \ldots, T^{n\epsilon+2}$ on $E_q$ (when divided by the square of the defining function of $E_q$) and vanishing to order $\geq 2$ on all exceptional divisors $E_r \neq E_q$,
\item sections in $H^0(\widetilde{V}, \widetilde{L}_n)$ restricting to (scalar multiples of) $S^{n\epsilon+1}, S^{n\epsilon}T, \ldots, T^{n\epsilon+1}$ on $E_q$ (when divided by the defining function of $E_q$) and vanishing to order $\geq 2$ on all exceptional divisors $E_r \neq E_q$, and
\item sections in $H^0(\widetilde{V}, \widetilde{L}_n)$ restricting on $E_q$ to (scalar multiples of) $S^{n\epsilon}, S^{n\epsilon-1}T, \ldots, T^{n\epsilon}$, a basis of $H^0(E_q, \mathcal{O}_{E_q}(-n\epsilon E_q))$ and vanishing to order $\geq 2$ on all exceptional divisors $E_r \neq E_q$.
\end{itemize}

\noindent Uniting a basis of $H^0(\widetilde{V}, \widetilde{L}_n \otimes \mathcal{O}_{\widetilde{V}}(-2\sum_{r=1}^k E_r)) \subset H^0(\widetilde{V}, \widetilde{L}_n)$ with suitable linear combinations of the sections above we obtain a set of sections $\sigma_0, \ldots, \sigma_N$ spanning $H^0(\widetilde{V}, \widetilde{L}_n)$ which can be subdivided in three disjoint parts for each $q = 1, \ldots, k$: In terms of the $(x,t)$ coordinates in one of the charts around $E_q$ the sections are either of the form
\[ \sqrt{\binom{n\epsilon}{j}} t^j + x^3 \cdot f_j(x,t),\ j = 0, \ldots, n\epsilon,\ \mathrm{or} \]
\[ x \cdot (\sqrt{\binom{n\epsilon+1}{l}} t^l + x \cdot g_l(x,t)),\ l = 0, \ldots, n\epsilon + 1,\ \mathrm{or}\ x^2 \cdot h(x,t), \]
where the $f_j$, $g_l$ and $h$ are regular functions in $x,t$ and there exists exactly one section of the respective form for each $j$ and each $l$. By multiplying with $s^{n\epsilon}$ and using $t \cdot s = 1$ and $x = sy$ we obtain expressions for the sections in the $(s,y)$-coordinates of the other chart around $E_q$, and these expressions in $s,y$ are completely similar to those in $x,t$.

\noindent Let $\widetilde{h}$ denote the hermitian metric on $\widetilde{L}_n$ and $\widetilde{\omega}$ the K\"ahler form on $\widetilde{V}$ induced by the sections $\sigma_0, \ldots, \sigma_N$. Using the coordinates $(x,t)$ around $E_q$ (the calculations are completely analogous when using the coordinates $(s,y)$ of the other chart around~$E_q$) the Taylor series expansion of $\log$ shows that 
\begin{eqnarray*} 
\widetilde{\omega}(0,0) & = & \left( \frac{i}{2\pi} \partial\overline{\partial} \log (\sum_{j=0}^N \sigma_j(x,t)\overline{\sigma_j(x,t)}) \right)(0,0) = 
\frac{1}{\pi} \left( dx \wedge d\overline{x} + n\epsilon dt \wedge d\overline{t} \right) = \\
 & = & \frac{1}{\pi} \sigma^\ast \omega_0(0,0) + n\epsilon p_2^\ast \tau_0(0,0), 
\end{eqnarray*}
because $F := \sum_{j=0}^N \sigma_j(x,t)\overline{\sigma_j(x,t)}$ is a power series in $x,\overline{x}, t, \overline{t}$, and the only terms of order $\leq 2$ in $F$ are $1, x\overline{x}, n\epsilon t\overline{t}$. Similarly in other points $P = (0,t_0) \in E_q$: Choose a unitary matrix 
$\begin{pmatrix} a & b \\ c & d \end{pmatrix} \in U(2)$ such that $t_0 = \frac{c}{a}$, and rewrite $F$ in terms of the coordinates $(x^\prime,t^\prime)$ given by
\[ x = x^\prime(a+bt^\prime),\ t = \frac{c+dt^\prime}{a+bt^\prime}. \]
Then $F \cdot |a+bt^\prime|^{2n\epsilon}$ is a power series in $x^\prime, \overline{x^\prime}, t^\prime, \overline{t^\prime}$,
and as before the only terms of order $\leq 2$ are $1, x^\prime\overline{x^\prime}, n\epsilon t^\prime\overline{t^\prime}$ because $\begin{pmatrix} a & b \\ c & d \end{pmatrix} \in U(2)$ implies that 
\[ (1+t\overline{t}) \cdot |a+bt^\prime|^2 = |a+bt^\prime|^2 + |c+dt^\prime|^2 = 1+t^\prime\overline{t^\prime}. \]
Since $\partial\overline{\partial}\log |a+bt^\prime|^2 = 0$ we conclude once again that
\[ \widetilde{\omega}(P) =  
    \frac{1}{\pi} dx^\prime \wedge d\overline{x^\prime} + n\epsilon dt^\prime \wedge d\overline{t^\prime} = 
    \frac{1}{\pi} \sigma^\ast \omega_0(P) + n\epsilon p_2^\ast \tau_0(P). \]

\noindent Finally, the statement on $\widetilde{\omega}_{|E_q}$ follows because $F_{|{x = 0}} = (1 + t\overline{t})^{n\epsilon}$.

\noindent \textit{Step 2.} On the tubular neighborhoods $\sigma^{-1}(U_q(\delta))$ we glue the metrics $\widetilde{h}$ on $\widetilde{L}_n$ and $h_{0,q}$ on $\widetilde{L}_{n|\sigma^{-1}(U_q(\delta))}$: To this purpose we use a partition of unity $(\overline{\rho}_1, \overline{\rho}_2)$ subordinate to the open cover ($\widetilde{V} - \sigma^{-1}(U_q(\delta/2)), \sigma^{-1}(U_q(\delta))$ of $\widetilde{V}$. We construct $\overline{\rho}_1, \overline{\rho}_2$ from a partition of unity $(\rho_1, \rho_2)$ subordinate to the open cover $(\mathbb{R} - (-\delta^2/4, \delta^2/4), (-\delta^2, \delta^2))$ of $\mathbb{R}$, by setting
\[ \overline{\rho}_i(x,t) := \rho_i(|\sigma(x,t)|^2) = \rho_i(|\sigma(s,y)|^2), i = 1,2.  \]
Note that we can choose $\rho_i$ such that the first-order partial derivatives of $\overline{\rho}_i$ are bounded by a constant multiple of $1/\delta$ and the second-order partial derivatives of $\overline{\rho}_i$ by a constant multiple of $1/\delta^2$.

\noindent Let $\sigma_0, \ldots, \sigma_N$ be the global sections of $\widetilde{L}_n$ constructed in Step 1. Then in coordinates $(x,t)$ around $E_q$ the metric $\widetilde{h}$ induced by these sections is given by 
\[ \widetilde{h}(\sigma(x,t)) = |\sigma(x,t)|/(\sum_{j=0}^N |\sigma_j(x,t)|^2)^{\frac{1}{2}} = |\sigma(x,t)| \cdot e^{-\frac{1}{2}\phi_1(x,t)}, \] with $\phi_1(x,t) = \log (\sum_{j=0}^N |\sigma_j(x,t)|^2)$, for each section $\sigma$ of $\widetilde{L}_n$. Similarly,
\[ h_{0,q}(\sigma(x,t)) = |\sigma(x,t)| \cdot e^{-\frac{1}{2}\phi_2(x,t)} \]
with $\phi_2 = n\epsilon \log(1+t\overline{t}) + \frac{1}{\pi} x\overline{x}(1+t\overline{t})$. Then the glued metric $\overline{h}$ can be constructed as
\[ \overline{h}(\sigma(x,t)) = |\sigma(x,t)| e^{-\frac{1}{2}(\overline{\rho}_1\phi_1 + \overline{\rho}_2\phi_2)} = |\sigma(x,t)|e^{-\frac{1}{2}(\phi_1 + \overline{\rho}_2(\phi_2 - \phi_1))}. \]
Its curvature is
\[ \frac{i}{2\pi}\left[ \partial\overline{\partial}\phi_1 + \partial\overline{\partial}(\overline{\rho}_2(\phi_2 - \phi_1)) \right] = 
   \frac{i}{2\pi}\left[ \partial\overline{\partial}\phi_1 + \partial(\overline{\rho}_2 \cdot \overline{\partial}(\phi_2 - \phi_ 1) + 
                                                                                                  \overline{\partial}\overline{\rho}_2 \cdot (\phi_2 - \phi_1)) \right] = \]
\[ = \frac{i}{2\pi}\left[ \partial\overline{\partial}\phi_1 + \partial\overline{\rho}_2 \cdot \overline{\partial}(\phi_2 - \phi_1) +
     \overline{\rho}_2 \cdot \partial\overline{\partial}(\phi_2 - \phi_1) + \partial\overline{\partial}\overline{\rho}_2 \cdot (\phi_2 - \phi_1) + \overline{\partial}\overline{\rho}_2 \cdot \partial(\phi_2 - \phi_1) \right]. \]
The Taylor series expansion of $\log$ and the properties of the sections $\sigma_j$ discussed in Step 1 show that $\phi_2 - \phi_1$ expands to a power series in $x, t$ only containing terms of order $\geq 3$. Hence the remarks on the partial derivatives of $\overline{\rho}_1$ and $\overline{\rho}_2$ imply that around $(x,t) = (0,0)$ all summands but the first converge everywhere on $\sigma^{-1}(U_q(\delta)) - \sigma^{-1}(U_q(\delta/2))$ to $0$ when $\delta$ tends to~$0$. Since $\frac{i}{2\pi}\partial\overline{\partial} \phi_1$ is strictly positive being the curvature of $\widetilde{h}$, it follows that $\overline{h}$ is a positive metric on $\widetilde{L}^n$ for $n$ sufficiently large and $\delta$ sufficiently small. Calling $\overline{\omega}$ the K\"ahler form obtained as the curvature of $\overline{h}$ we have that
\[ \overline{\omega}_{|\sigma^{-1}(U_q(\delta/2))} = \omega_{0,q}. \]

\noindent \textit{Step 3.} We glue in standard K\"ahler balls of radius $\sqrt{\frac{\epsilon}{\pi}}$ to $(\widetilde{V}, \frac{1}{n}\overline{\omega})$ replacing the exceptional divisors $E_q$: Let $\mathcal{L}(r)$ denote the preimage of the ball $B(r)$ centered in $0 \in \mathbb{C}^2$ under the standard blow-up $\sigma$ of $\mathbb{C}^2$ in $0$ and let $\rho(\delta,\epsilon)$ denote the K\"ahler form
\[ \rho(\delta,\epsilon) := \delta \cdot \sigma^\ast \omega_0 + \epsilon \cdot p_2^\ast \tau_0 \]
on $\mathcal{L}(r)$, for $\delta, \epsilon > 0$. The construction of $\overline{\omega}$ implies that an appropriate rescaling of the $(x,y)$-coordinates around $x_q$ without changing the homogeneous coordinates $S,T$ on $E_q$ yields holomorphic embeddings
\[ \phi_q: \mathcal{L}(1 + \epsilon_q) \hookrightarrow \widetilde{V} \]
such that $\phi_q^\ast(\frac{1}{n}\overline{\omega}) = \rho(\delta_q,\epsilon)$, for some $\epsilon_q, \delta_q > 0$. The symplectic blow-down construction in \cite[\S 5.4, in particular \S 5.4.A]{McDP94} shows that there exist a K\"ahler form on $V$ representing $c_1(L)$ and a K\"ahler embedding of $k$ standard balls of radii $\sqrt{\frac{\epsilon}{\pi}}$ into $V$, wrt this K\"ahler form.

\vspace{0.2cm}

\noindent The opposite direction of Thm.~\ref{packing-thm} follows immediately by using symplectic blow up constructions on K\"ahler manifolds as described in \cite[\S 5.3, in particular \S 5.3.A]{McDP94} (see also \cite[Lem.5.3.17]{LazPAG1}). \hfill $\Box$

\begin{rem}
A proof that is just notationally more involved will show the analogous theorem for projective complex manifolds of arbitrary dimension and their multi-point Seshadri constants. For the one-point case in arbitrary dimensions see the results of  Witt Nystr\"om   
 \cite{WN15a, WN15b}.
\end{rem}

\noindent We construct the K\"ahler form $\omega$ and the Fubini-Study K\"ahler packing postulated in Cor.~\ref{FS-packing-cor} from the K\"ahler form and the standard K\"ahler packing appearing in Thm.~\ref{packing-thm} by suitably modifying the symplectomorphism between flat K\"ahler balls and Fubini-Study balls given by
\[ \phi: (B_0(1), \omega_{\mathrm{std}}) \rightarrow (\mathbb{C}^2, \omega_{\mathrm{FS}}),\ \  
   (z_1, z_2) \mapsto \frac{1}{(1- \sum_{i=1}^2 |z_i|^2 )^\frac{1}{2}} \cdot (z_1, z_2) \]
(see \cite[Ex.7.14]{McDS95}) and its inverse, to glue in a Fubini-Study ball into a flat K\"ahler ball and vice versa:

\begin{lem} \label{FS-std-glue-lem}
For all $R, \epsilon, \lambda > 0$ there is a K\"ahler form $\tau = \tau(R, \epsilon, \lambda)$ on $\mathbb{C}^2$ such that
\[ \tau_{|B_0(R)} = \lambda^2 \cdot \frac{1}{R^2+1} \omega_{\mathrm{std}}\ \mathrm{and}\ 
   \tau_{|\mathbb{C}^2 - B_0(R+\epsilon)} = \lambda^2 \omega_{\mathrm{FS}}. \]
For all $0 < \epsilon < \frac{1}{2}$, $\lambda > 0$ there is a K\"ahler form $\sigma$ on $B_0(1)$ such that
\[ \sigma_{|B_0(1-2\epsilon)} = \lambda^2 \cdot \frac{1}{4\epsilon(1-\epsilon)} \omega_{\mathrm{FS}}\ \mathrm{and}\ 
   \sigma_{|B_0(1) - B_0(1-\epsilon)} = \lambda^2 \omega_{\mathrm{std}}. \]
\end{lem}
\begin{proof}
We obtain $\tau$ as the pull back of $\lambda^2 \omega_{\mathrm{FS}}$ via the monotone embedding given in polar coordinates $u \in S^3, r \in (0, \infty)$ on $\mathbb{C}^n - \{0\}$ by
\[ \begin{array}{rclccl}
    (u,r) & \mapsto & (u, r/R) & \mapsto & (u, r(R^2-r^2)^{-\frac{1}{2}}) & \mathrm{on}\ B_0(R) - \{0\}, \\
    z & \mapsto & z & & & \mathrm{on}\ \mathbb{C}^n - B_0(R+\epsilon)
    \end{array} \]
and smoothened on $B_0(R+\epsilon) - B_0(R)$. Then $\tau$ is a K\"ahler form by \cite[\S 5.1]{McDP94} and satisfies the requested properties.

\noindent Similarly, we obtain $\sigma$ as the pull back of $\lambda^2 \omega_{\mathrm{std}}$ via the monotone embedding given by
\[ \begin{array}{rclccl}
    (u,r) & \mapsto & \left(u, \frac{r}{2(\epsilon(1-\epsilon))^{\frac{1}{2}}}\right) & \mapsto & 
    \left(u, \frac{r}{(r^2+2(\epsilon(1-\epsilon))^{\frac{1}{2}})^{\frac{1}{2}}}\right) & \mathrm{on}\ B_0(1-2\epsilon) - \{0\}, \\
    z & \mapsto & z & & & \mathrm{on}\ B_0(1) - B_0(1-\epsilon)
    \end{array} \]
and smoothened on $B_0(1-\epsilon) - B_0(1-2\epsilon)$.
\end{proof}

\noindent Note that rescaling yields symplectomorphisms 
\[ (B_0(R), \lambda^2 \cdot \frac{1}{R^2+1} \omega_{\mathrm{std}}) \rightarrow 
   (B_0(\frac{R}{(R^2+1)^{\frac{1}{2}}}), \lambda^2 \omega_{\mathrm{std}}) \]
and 
\[ (B_0(1-2\epsilon), \lambda^2 \cdot \frac{1}{4\epsilon(1-\epsilon)} \omega_{\mathrm{FS}}) \rightarrow 
   (B_0(\frac{1-2\epsilon}{2(\epsilon(1-\epsilon))^\frac{1}{2}}), \lambda^2 \omega_{\mathrm{FS}}). \]

\noindent Then Cor.~\ref{FS-packing-cor} is an immediate consequence of Thm.~\ref{packing-thm} and Lem.~\ref{FS-std-glue-lem}. Note that $\omega$ represents the same first Chern class as the K\"ahler form constructed in Thm.~\ref{packing-thm} because $\omega$ is constructed as its pullback via a map homotopic to identity.

\section{K\"ahler packings on toric surfaces} \label{toric-sec}

\noindent To state exact results we fix notations and recall some facts on toric varieties following \cite{FulTV}: $N \cong \mathbb{Z}^n$ denotes a lattice of rank $n$, $M := \mathrm{Hom}_{\mathbb{Z}}(N, \mathbb{Z})$ the dual lattice of $N$, $\Delta$ a fan of rational strongly convex polyhedral cones $\sigma \subset N_{\mathbb{R}} := N \otimes_{\mathbb{Z}} \mathbb{R}$ and $X(\Delta)$ the $n$-dimensional toric variety associated to $\Delta$, together with the natural action of the torus $T_N \cong (\mathbb{C}^\ast)^n$ on $X(\Delta)$.

\noindent A toric variety $X(\Delta)$ is covered by affine toric variety $U_\sigma := \mathbb{C}[\sigma^\vee \cap M]$, $\sigma \in \Delta$, where $\sigma^\vee = \{u \in M_{\mathbb{R}}: \langle u, v \rangle \geq 0\ \mathrm{for\ all\ } v \in \sigma\}$ is the dual cone to $\sigma$ in $M_{\mathbb{R}}$ and $\sigma^\vee \cap M$ is a semigroup in $M$.

\noindent A toric variety $X(\Delta)$ is complete if the support $|\Delta| = \bigcup_{\sigma \in \Delta} \sigma$ covers all of $N_{\mathbb{R}}$.

\noindent $T_N$-invariant morphisms $X(\Delta^\prime) \rightarrow X(\Delta)$ between two $n$-dimensional toric varieties correspond to abelian group homomorphisms $\alpha: N \rightarrow N$ such that $\alpha_{\mathbb{R}} : = \alpha \otimes_{\mathbb{Z}} \mathbb{R}: N_{\mathbb{R}} \rightarrow N_{\mathbb{R}}$ maps each cone of $\Delta^\prime$ into a cone of $\Delta$.

\noindent Cones $\sigma \in \Delta$ of maximal dimension $n$ correspond to $T_N$-fixed points $x_\sigma \in X(\Delta)$ whereas cones in $\Delta$ of lower dimension correspond to higher-dimensional $T_N$-orbits in $X(\Delta)$. In particular the rays $\tau \in \Delta$ correspond to $(n-1)$-dimensional $T_N$-orbits whose closures are the irreducible $T_N$-stable Weil divisors on $X(\Delta)$. For a ray $\tau \in \Delta$ let $v_\tau \in N$ denote the generator of $\tau$ in $N$, and $D_\tau$ the Weil divisor corresponding to $\tau$.

\noindent A $T_N$-stable Cartier divisor $D$ on $X(\Delta)$ is defined by elements $u_D(\sigma) \in M$ for each $\sigma \in \Delta$ of maximal dimension $n$ such that $u_D(\sigma) - u_D(\sigma^\prime) \in (\sigma \cap \sigma^\prime)^\perp$. The corresponding Weil divisor is given as $D = -\sum_{\sigma \supset \tau \in \Delta\ \mathrm{ray}} \langle u_D(\sigma), v_\tau \rangle D_\tau$.

\noindent $X(\Delta)$ is nonsingular if each cone $\sigma \in \Delta$ is generated by $n$ vectors $v_1, \ldots, v_n \in N$ that are a $\mathbb{Z}$-basis of $N$. In that case the $T_N$-stable Weil divisors coincide with the $T_N$-stable Cartier divisors. Note also that the dual cone $\sigma^\vee \subset M_{\mathbb{R}}$ and the semigroup $\sigma^\vee \cap M \subset M$ are generated by a $\mathbb{Z}$-basis of $M$ if $\sigma$ is generated by a $\mathbb{Z}$-basis.

\begin{prop}[{\cite[Sec.2.4]{FulTV}}] \label{toric-blowup-prop}
Let $X(\Delta)$ be a nonsingular toric variety and $\sigma \in \Delta$ a cone of maximal dimension corresponding to the $T_N$-fixed point $x_\sigma$. 

\noindent Then the blow up of $X(\Delta)$ in $x_\sigma$ is given by the morphism $X(\Delta^\prime) \rightarrow X(\Delta)$ where $\Delta^\prime$ is constructed from $\Delta$ by subdividing $\sigma$ into $n$ cones $\sigma_i$ generated by
\[ v_1, \ldots, v_{i-1}, v_1 + \cdots + v_n, v_{i+1}, \ldots, v_n \]
where $v_1, \ldots, v_n \in N$ are spanning $\sigma$ and also are a $\mathbb{Z}$-basis of the lattice $N$.

\noindent The exceptional divisor on $X(\Delta^\prime)$ is $T_N$-stable and corresponds to the ray $\tau$ generated by $v_1 + \cdots + v_n$. \hfill $\Box$
\end{prop}

\noindent A $T_N$-stable Cartier divisor $D = \sum_{\tau \in \Delta\ \mathrm{ray}} a_\tau D_\tau$ on $X(\Delta)$ defines a rational convex polyhedron in $M_{\mathbb{R}}$, as
\[ P_D = \{ u \in M_{\mathbb{R}}: \langle u, v_\tau \rangle \geq - a_\tau\ \mathrm{for\ all\ rays\ } \tau \in \Delta \}. \]
The elements of $P_D \cap M$ correspond to $T_N$-stable generators of the space of global sections of $\mathcal{O}_{X(\Delta)}(D)$.

\noindent $D$ is ample if and only if the elements $u_D(\sigma) \in M$ describing $D$ are exactly the vertices of $P_D$ (see \cite[p.70]{FulTV}). If (and only if) such an ample $T_N$-stable divisor exists on $X(\Delta)$ and $X(\Delta)$ is complete then the toric variety $X(\Delta)$ is projective.

\noindent The following two results can be found in \cite[\S 4]{BDH+09} but we present the proof as a convenience to the reader and because some details are needed later on.

\begin{prop}[{\cite[\S 4]{BDH+09}}] \label{ample-blowup-prop}
Let $X(\Delta)$ be an $n$-dimensional non-singular projective toric variety, $\sigma \in \Delta$ a cone of maximal dimension $n$ with corresponding $T_N$-fixed point $x_\sigma$ and $\pi: X(\Delta^\prime) \rightarrow X(\Delta)$ the blow up of $X(\Delta)$ in $x_\sigma$, with exceptional divisor $E_\sigma$, as constructed in Prop.~\ref{toric-blowup-prop}. Let $D$ be an ample $T_N$-stable Cartier divisor on $X(\Delta)$ with associated polyhedron $P_D$. 
\begin{itemize}
\item[(a)] Let $v_1, \ldots, v_n \in N$ be the generators of the edges of $\sigma$ and $w_1, \ldots, w_n \in M$ the generators of the edges of $\sigma^\vee$. If $\sigma^\prime \in \Delta$ is a cone of maximal dimension $n$ intersecting $\sigma$ in the facet spanned by $v_1, \ldots, v_{i-1}, v_{i+1}, \ldots, v_n$ then the vertices $u_D(\sigma)$ and $u_D(\sigma^\prime)$ of $P_D$ differ by a multiple $\epsilon_i w_i$ of $w_i$, $\epsilon_i > 0$.
\item[(b)]$D_\epsilon := \pi^\ast D - \epsilon E_\sigma$ is an ample ($\mathbb{Q}$-)divisor if and only if $\epsilon < \min_{i=1, \ldots, n} \epsilon_i$, and its associated polyhedron $P_{D_\epsilon}$ is obtained from $P_D$ by taking away the simplex with vertex $u_D(\sigma)$ and edges $\epsilon w_i$ starting in $u_D(\sigma)$. 
\end{itemize}
\end{prop}
\begin{proof}
Let $v_i^\prime$ span the one edge $\tau^\prime$ of $\sigma^\prime$ that is not an edge of $\sigma$. In particular, $v_i^\prime$ and $v_i$ lie on different sides of the hyperplane spanned by $v_1, \ldots, v_{i-1}, v_{i+1}, \ldots, v_n$. Consequently, $\langle w_i, v_i^\prime \rangle < 0$.

\noindent If $D = \sum_{\tau \in \Delta\ \mathrm{ray}} a_\tau D_\tau$ then $u_D(\sigma) = - \sum_{i=1}^n a_iw_i$ because $\langle u_D(\sigma), v_i \rangle = -a_i$, and the $w_i$ are a $\mathbb{Z}$-basis of $M$ dual to the $\mathbb{Z}$-basis $v_i$ of $N$. Since $D$ is ample we must have $\langle u_D(\sigma), v_i^\prime \rangle > -a_{\tau^\prime}$

\noindent All these facts imply that there is $\epsilon_i > 0$ such that $\langle u_D(\sigma) + \epsilon_i w_i, v_i^\prime \rangle = -a_{\tau^\prime}$ whereas  $\langle u_D(\sigma) + \epsilon_i w_i, v_j \rangle = -a_j$ for $j = 1, \ldots, i-1, i+1, \ldots, n$. By the same argument as before, we may conclude $u_D(\sigma) + \epsilon_i w_i = u_D(\sigma^\prime)$, and (a) is shown.

\noindent For (b) note that $w_1 - w_i, \ldots, w_i, \ldots, w_n - w_i$ generate the semigroup $\sigma_i^\vee \cap M$ as these elements are a dual basis to $v_1, \ldots, v_1+\cdots+v_n, \ldots, v_n$. Furthermore,
\[ D_\epsilon = \pi^\ast D - \epsilon E_\sigma = \sum_{\tau \in \Delta\ \mathrm{ray}} a_\tau D_\tau + (\sum_{\tau \in \Delta\ \mathrm{ray}} a_\tau) E_\sigma - \epsilon E_\sigma. \]
Consequently, $\sigma_i$ corresponds to the vertex of $P_{D_\epsilon}$ given by
\[ -\sum_{k=1, k \neq i}^n a_k(w_k - w_i) - (\sum_{k=1}^n a_k)w_i + \epsilon w_i =  -\sum_{k=1}^n a_kw_k + \epsilon w_i = u_D(\sigma) + \epsilon w_i. \]
So $D_\epsilon$ is ample if $\epsilon < \epsilon_i$ for all $i=1, \ldots, n$, and the polyhedron $P_{D_\epsilon}$ replaces the vertex $u_D(\sigma)$ of $P_D$ by the vertices $u_D(\sigma) + \epsilon w_i$, cutting off the simplex as described. 
\end{proof}

\begin{cor} \label{toric-Sesh-cor} 
Let $X(\Delta)$ be an $n$-dimensional non-singular projective toric variety, let $\pi: X(\Delta^\prime) \rightarrow X(\Delta)$ be the blow up of $X(\Delta)$ in several $T_N$-fixed points $x_{\sigma_1}, \ldots, x_{\sigma_k}$, with exceptional divisors $E_k$, and let $D$ be an ample $T_N$-stable Cartier divisor on $X(\Delta)$. Then $\pi^\ast D - \epsilon \sum_{l=1}^k E_k$ is an ample ($\mathbb{Q}$)-divisor if, and only if,
\[ \epsilon < \frac{1}{2} \min_{1 \leq l \leq k, 1 \leq i \leq n} \epsilon_i^l, \]
where the $\epsilon_i^l > 0$ are those numbers determined for each edge $\tau_i$ of the cones $\sigma_l$ in Prop.~\ref{ample-blowup-prop}. \hfill $\Box$ 
\end{cor}

\begin{exm} \label{toric-ex}
Consider the nonsingular toric projective variety $\mathbb{P}^2_{\mathbb{C}}$, on whom the torus $T_N \cong (\mathbb{C}^\ast)^2$ acts as $(s,t) \cdot [X:Y:Z] = [sX:tY:Z]$. The three cones of maximal dimension in the fan $\Delta$ describing $\mathbb{P}^2_{\mathbb{C}} = X(\Delta)$ are separated by the rays $\tau_X$, $\tau_Z$ and $\tau_Y$ spanned by $v_X = (1,0)$, $v_Z = (-1,-1)$ and $v_Y = (0,1)$ in $N$, respectively. $\sigma_Z$, $\sigma_Y$ and $\sigma_X$ correspond to the three $T_N$-fixed points $x_Z = [0:0:1]$, $x_Y = [0:1:0]$ and $x_X = [1:0:0]$ in $\mathbb{P}^2_{\mathbb{C}}$, respectively. The rays $\tau_X$, $\tau_Z$ and $\tau_Y$ correspond to the $T_N$-stable divisors $D_X = \{X = 0\}$, $D_Z = \{Z = 0\}$ and $D_Y = \{Y = 0\}$, respectively. These are lines in $\mathbb{P}^2_{\mathbb{C}}$, hence linearly equivalent divisors. For an integer $k > 0$ the moment polytope $P_D$ of the divisor $D := kD_Z$ is 
\[ P_D = \{ (u_1, u_2) \in \mathbb{R}^2: u_1, u_2 \geq 0, u_1 + u_2 \leq k \}. \]

\begin{minipage}{6cm}{
\includegraphics[width=45mm]{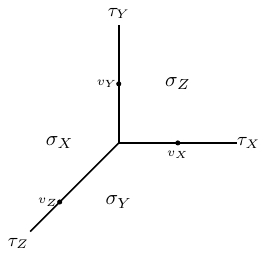}}
\end{minipage}
\begin{minipage}{6cm}{
\includegraphics[width=45mm]{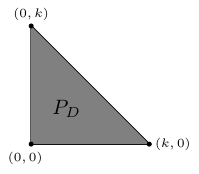}}
\end{minipage}
%\caption{The fan describing $\mathbb{P}^2_{\mathbb{C}}$ and the moment polytope $P_D$}

\noindent Blowing up the $T_N$-fixed points $x_Z$, $x_Y$ and $x_X$ yields a toric variety $\widetilde{X} = X(\widetilde{\Delta})$ with fan $\widetilde{\Delta}$ obtained from $\Delta$ by splitting up the cones $\sigma_Z$, $\sigma_Y$ and $\sigma_X$ with rays spanned by $v_{XY} = (1,1)$, $v_{ZX} = (0,-1)$ and $v_{YZ} = (-1,0)$, respectively. These rays correspond to the exceptional divisors $E_Z$, $E_Y$ and $E_X$, respectively, of the blow up morphism $\pi: \widetilde{X} \rightarrow \mathbb{P}^2_{\mathbb{C}}$. The rays $\tau_X$, $\tau_Z$ and $\tau_Y$ correspond to the strict $\pi$-transforms of $D_X$, $D_Z$ and $D_Y$, respectively. Consequently, for $\widetilde{D} = k\pi^\ast D_Z - lE_X - lE_Z - lE_Y$ the moment polytope $P_{\widetilde{D}}$ is given as
\[ P_{\widetilde{D}} = \{ (u_1, u_2) \in \mathbb{R}^2: 0 \leq u_1  \leq k-l, 0 \leq u_2  \leq k-l, l \leq u_1 + u_2 \leq k  \}, \] 
and $\widetilde{D}$ is ample if and only if $0 < l < \frac{k}{2}$. This implies that 
\[ \epsilon(D; x_Z, x_Y, x_X) = \frac{1}{2}. \]

\begin{minipage}{6cm}{
\includegraphics[width=45mm]{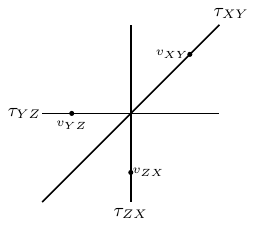}}
\end{minipage}
\begin{minipage}{6cm}{
\includegraphics[width=45mm]{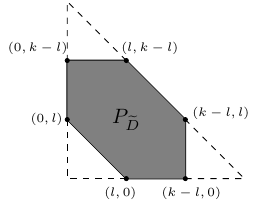}}
\end{minipage}
\end{exm}

\noindent We now construct approximations to K\"ahler packings on toric varieties using $T_N$-stable global sections of high enough multiples of $L$. We also investigate how the toric symplectic moment maps induced by these global sections pull back to the embedded balls. 

\noindent Recall that on an $n$-dimensional projective toric variety $X$ with very ample divisor $D$, a toric moment map is given by
\[ \mu: X \rightarrow 
    \frac{1}{\sum_{u \in P_D \cap M} |x^u|^2} \cdot \sum_{x \in P_D \cap M} |x^u|^2 \cdot u \in \mathbb{R}^n \]
where the $x^u \in H^0(X, \mathcal{O}_X(D))$ are a basis of $T_N$-stable global sections. Then $\mu(X) = P_D$ \cite[Ch.4.2]{FulTV} and $\mu$ is a symplectic moment map for the $S_N$-action on $X$  where $S_N \subset T_N$ is the real torus subgroup of $T_N$ given by points $(z_1, \ldots, z_n)$ with $|z_i| = 1$ \cite[Ex.5.48]{McDS95}. These properties do not change when we multiply the $x^u$ with arbitrary constants $c_u \in \mathbb{C}^\ast$. Furthermore, 
\[ \mu_{\mathrm{std}}: B_0(r) \rightarrow \mathbb{R}^n, (z_1, \ldots, z_n) \mapsto (|z_1|^2, \ldots, |z_n|^2) \]
is a symplectic moment map for the standard $S_N$-action on a ball $B_0(r)$.

\begin{thm} \label{approx-packing-thm}
Let $X(\Delta)$ be a nonsingular projective toric surface and $\pi: X(\Delta^\prime) \rightarrow X(\Delta)$ the blow-up of $T_N$-fixed points $x_\sigma$ corresponding to $2$-dimensional cones $\sigma \in \Delta$, with $E_\sigma \subset X(\Delta^\prime)$ the exceptional divisor over $x_\sigma$.

\noindent Let $L$ be an ample divisor over $X(\Delta)$ and let $0 < \epsilon \in \mathbb{Q}$ such that $L_\epsilon := \pi^\ast L - \epsilon \cdot \sum_\sigma E_\sigma$ is ample on $X(\Delta^\prime)$. Then for $k \gg 0$ sufficiently divisible and $\delta > 0$ there exist $T_N$-stable global sections $s_0^{(\delta)}, \ldots, s_{N_k}^{(\delta)} \in H^0(X(\Delta), L^{\otimes k\epsilon})$ inducing the K\"ahler form $\omega_\delta$ on $X(\Delta)$ and the moment map $\mu_\delta: X(\Delta) \rightarrow \mathbb{R}^2$, and there exist embeddings $\phi_\sigma^{(\delta)}: B_0(\sqrt{\frac{\epsilon}{\pi}}) \rightarrow X(\Delta)$ with $\phi_\sigma^{(\delta)}(0) = x_\sigma$ such that
\begin{enumerate}
\item $\phi_\sigma^{(\delta)\ast} \omega_\delta$ approximates $\omega_{\mathrm{std}}$ on $B_0(\sqrt{\frac{\epsilon}{\pi}})$  for $\delta \rightarrow 0$, and
\item $\mu_\delta \circ \phi_\sigma^{(\delta)}: B_0(\sqrt{\frac{\epsilon}{\pi}}) \rightarrow \mathbb{R}^2$ approximates the standard moment map $\mu_{\mathrm{std}}$ on $B_0(\sqrt{\frac{\epsilon}{\pi}})$  for $\delta \rightarrow 0$.
\end{enumerate}
\end{thm}
\begin{proof}
Let $D$ be a $T_N$-stable Cartier divisor on $X(\Delta)$ such that $L = \mathcal{O}_{X(\Delta)}(D)$, and choose a sufficiently divisible $k \gg 0$. For a $2$-dimensional cone $\sigma \in \Delta$ the two generators $w_1, w_2 \in M$ of the edges of $\sigma^\vee \cap M$ correspond to affine coordinates $z_1, z_2$ on $U_\sigma \cong \mathrm{Spec}\ \mathbb{C}[\sigma^\vee \cap M] \cong \mathbb{A}^2_{\mathbb{C}}$ centered in $x_\sigma$. Then all the $T_N$-stable global sections of $L^{\otimes k\epsilon}$ can be written as monomial terms $c_u z^u$, with $c_u \in \mathbb{C}$ and $u \in (P_{k\epsilon D} - u_{k\epsilon D}(\sigma)) \cap M$.

\noindent Since $\epsilon < \epsilon(L; x_1, \ldots, x_k)$ Cor.~\ref{toric-Sesh-cor} implies that global sections $z_1^az_2^b$ with $0 \leq a+b \leq k\epsilon$ do not coincide with global sections of that form but with respect to affine coordinates on $U_{\sigma^\prime}$, $\sigma^\prime$ another $2$-dimesnional cone in $\Delta$. 

\noindent For each $2$-dimensional cone $\sigma \in \Delta$ we choose the coefficient $c_{a,b}$ of the monomial $z_1^az_2^b$, $0 \leq a+b \leq k\epsilon$, to be the square root of the coefficient of the monomial $|z_1|^{2a}|z_2|^{2b}$ in $(\delta^2 + |z_1|^2 + |z_2|^2)^{k\epsilon}$. For all the other $T_N$-stable global sections we choose the coefficient to be $1$.

\noindent The $T_N$-stable global sections provided with these coefficients induce a K\"ahler form $\omega_\delta$ whose restriction to $U_\sigma$ is
\[ \omega_{\sigma|U_\sigma} = \frac{1}{k} \cdot \frac{i}{2\pi} \partial\overline{\partial}\log \left( (\delta^2 + |z_1|^2 + |z_2|^2)^{k\epsilon} +\ \mathrm{terms\ in\ } |z_1|^2, |z_2|^2\ \mathrm{of\ order\ } > k\epsilon \right) \]
and a moment map whose restriction to $U_\sigma$ is
\[ \mu_{\delta|U_\sigma}(z) = \frac{1}{k} \cdot \frac{1}{\sum_{u \in P_{k\epsilon D} - k\epsilon u_D(\sigma)} |c_u|^2 |z|^{2u}}
    \sum_{u \in P_{k\epsilon D} - k\epsilon u_D(\sigma)} |c_u|^2 |z|^{2u} \cdot u. \]
For the embedding
\[ \phi_{\delta, R}: B_0(R) \rightarrow U_\sigma \subset X(\Delta), z \mapsto \delta \cdot z \]
we obtain that
\[ \begin{array}{rcl}
    \phi_{\delta, R}^\ast \omega_\delta & = & \frac{1}{k} \cdot \frac{i}{2\pi} \partial\overline{\partial}\log 
    \left( \delta^{2k\epsilon}(1 + |z_1|^2 + |z_2|^2)^{k\epsilon} +\ \mathrm{terms\ in\ of\ order\ } > k\epsilon\ \mathrm{in\ } \delta^2    
    \right) \\ 
      & \stackrel{\delta \rightarrow 0}{\longrightarrow} & \frac{1}{k} \cdot \frac{i}{\pi} \cdot k\epsilon \cdot         
         \partial\overline{\partial}\log (1 + |z_1|^2 + |z_2|^2) = \epsilon \cdot \omega_{\mathrm{FS}}. 
   \end{array} \]
Similarly, $\mu_\delta \circ \phi_{\delta, R}$ tends to 
\[ z \mapsto \frac{1}{k} \cdot \frac{1}{(1 + |z_1|^2 + |z_2|^2)^{k\epsilon}} \cdot 
                    \sum_{|u| \leq k\epsilon} \frac{|c_u|^2}{\delta^{2(k\epsilon - |u|)}} |z|^{2u} \cdot u \]
for $\delta \rightarrow 0$, and that is the toric moment map generated by the global sections $\binom{k\epsilon}{u_1+u_2} \binom{u_1+u_2}{u_1}z_1^{u_1}z_2^{u_2}$, $0 \leq u_1+u_2 \leq k\epsilon$, hence the symplectic moment map with respect to $\epsilon \cdot \omega_{\mathrm{FS}}$.

\noindent Rescaling the symplectomorphism 
$\phi: (B_0(1), \omega_{\mathrm{std}}) \rightarrow (\mathbb{C}^n,\omega_{\mathrm{FS}})$ discussed before Lem.~\ref{FS-std-glue-lem} and noting that $\phi$ is $S_N$-invariant we deduce properties (1) and (2).
\end{proof}

\noindent Note that the limits of $\omega_\delta$ and $\mu_\delta$ on $X(\Delta)$ do not exist, as the embeddings degenerate to maps onto points. Instead one needs the techniques in the proofs of Thm.~\ref{packing-thm} and Cor.~\ref{FS-packing-cor} to glue in flat resp. Fubini-Study balls.

%\bibliographystyle{alpha}

%\bibliography{/Forschung/Mathematik/doktor}
%\bibliography{doktor}

\newcommand{\etalchar}[1]{$^{#1}$}
\def\cprime{$'$}

\end{document}